\documentclass[11pt, letterpaper]{amsart}
\usepackage{graphicx}
\usepackage{amssymb,amsbsy,amsmath,amsfonts,amssymb,amscd}
\usepackage[hyperfootnotes=false, colorlinks=true, linkcolor=blue]{hyperref}
\usepackage{mathrsfs}
\usepackage{epsfig, graphicx}
\usepackage{graphics,epsfig}
\usepackage{mathrsfs}
\usepackage{amsthm}
\usepackage{mathtools}
\usepackage[dvipsnames]{xcolor}
\usepackage{changes}
\usepackage[margin=1in]{geometry}
\title{Sharp Quantitative Convergence of Compactly Supported Aggregation-Diffusion}
\author[Yunhong Deng]{Yunhong Deng}
\address{Department of Mathematics, Statistics and Computer Science, University of Illinois, Chicago, Illinois 60607, USA}
\email{ydeng43@uic.edu}


\date{\today}

\newtheorem{theorem}{Theorem}
\newtheorem{lemma}[theorem]{Lemma}

\newtheorem{proposition}[theorem]{Proposition}

\theoremstyle{definition}

\theoremstyle{definition}
\newtheorem{remark}{Remark}
\theoremstyle{definition}

\theoremstyle{definition}
\newtheorem{assumption}{Assumption}

\newcommand{\ud}{\,\mathrm{d}}

\begin{document}
\begin{abstract}
We establish the quantitative convergence of nonlinear aggregation-diffusion equations in dimension one, with sublinear attraction kernel and degenerate diffusion higher than quadratic. In these models, steady state is known to be the unique, compactly supported, and radial decreasing probability distribution. In this paper, we work with compactly supported densities in dimension one, we prove these solutions converge to the steady state exponentially fast.
\end{abstract}

\maketitle
\section{Introduction}
We consider aggregation-diffusion equations (where we focus on $d = 1$ throughout the paper)
\begin{equation}\label{0-1}
    \frac{\partial \varrho}{\partial t} - \Delta \varrho^{m} - \mathrm{div}(\varrho \nabla W \ast \varrho) = 0 \text{ on } \mathbb{R}^{d}.
\end{equation}
It describes the particle dynamics under two competing effects, a non-local attraction modeled by $W$ and a local repulsion given by the degenerate diffusion $\Delta \varrho^{m}$.
We consider degenerate diffusion with exponent $m > 2$, and a smooth, radial increasing, and sublinear kernel $W$ with growth rate $\frac{1}{2} + \varepsilon$ if $\varrho_{0}$ is symmetric and compactly supported, or $\frac{1}{2} + \frac{1}{2m} + \varepsilon$ for a general compactly supported initial value, where $\varepsilon > 0$.

Over the last three decades, the optimal mass transportation approach \cite{JKO, otto2001geometry, carrillo2003kinetic, ambrosio2008gradient, villani2008optimal, santambrogio2015optimal} has become a central tool for studying~\eqref{0-1}, which can be interpreted as the gradient flow of the free energy functional
\begin{equation}\label{0-2-energy}
    \mathcal{F}(\varrho) := \frac{1}{m - 1} \int_{\mathbb{R}^d} \varrho^{m} \ud x + \frac{1}{2} \int_{\mathbb{R}^d} \varrho \times (W \ast \varrho) \ud x,
\end{equation}
in the Wasserstein space of probability measures $(\mathscr{P}_{2}(\mathbb{R}^d), W_{2})$ where
\begin{equation*}
    W_{2}(\mu, \nu) := \inf \bigg\{\int_{\mathbb{R}^d} |S(x) - x|^2 \ud \mu : S\# \mu = \nu\bigg\}^{\frac{1}{2}},
\end{equation*}
is the Wasserstein distance. This gradient flow formulation can be partially seen by rewriting~\eqref{0-1} as the following continuity equation which describes particles transporting along the negative Wasserstein gradient $D_{W}\mathcal{F}(\varrho)$:
\begin{equation*}
    \frac{\partial \varrho}{\partial t} - \mathrm{div}(\varrho D_{W}\mathcal{F}(\varrho)) = 0,\quad \text{ where } D_{W}\mathcal{F}(\varrho) = \frac{m}{m - 1} \nabla \varrho^{m - 1} + \nabla W \ast \varrho.
\end{equation*}
For an introduction of the Wasserstein gradient flow formulation, we refer the reader to the previous work \cite{deng2026quantitative} and the references therein.

The optimal transport method is physically intuitive and naturally adapted to the particle dynamics~\eqref{0-1}. Recently, it revealed rich theories in studies of statistical machine learning, such as the training dynamics of shallow neural networks \cite{arbel2019maximum, rotskoff2022trainability, fernandez2022continuous, chizat2026quantitative}, self-attention dynamics in transformers \cite{rigollet2025mean, geshkovski2025mathematical, castin2025unified}, sampling methods \cite{liu2019understanding, sun2023discrete}, and optimizing probability measures \cite{wang2020information, li2024stochastic, di2025inexact, seo2026local, lascu2026saddle}.

\subsection{Established methods for long-time asymptotics}
Understanding the long-time behavior of solutions is always a central problem for~\eqref{0-1}. Since the inception of the theory, the Wasserstein gradient flow formulation was invented by Otto~\cite{otto2001geometry} to explain the long-time asymptotics of porous media diffusion, to be the contraction of self-similar Barenblatt profiles. This methodology for general nonlinear diffusion was later fully developed in \cite{denzler2005fast, seis2014long}.

In the early developments of gradient flow theories in Wasserstein space, the displacement convexity principle \cite{mccann1997convexity} was a bedrock for proving the long-time convergence. In their seminal work, Carrillo--McCann--Villani \cite{carrillo2003kinetic} proved the quantitative long-time convergence for a wide class of granular media flows, based on the Bakry-\'Emery method \cite{bakry2006diffusions} and Otto-Villani's log-Sobolev inequality \cite{otto2000generalization}, in which one leverages displacement convexity (or Bakry--\'Emery curvature bound) along the gradient flow to derive functional inequalities that encode energy-dissipation relation:
\begin{equation*}
    \underbrace{\mathcal{F}(\varrho) - \mathcal{F}(\varrho_{\ast}) \le \frac{1}{2\lambda}\|D_{W}\mathcal{F}(\varrho)\|^{2}_{L^{2}(\varrho)}}_{\text{log-Sobolev inequality}},\quad \underbrace{\mathcal{F}(\varrho) - \mathcal{F}(\varrho_{\ast}) \le W_{2}(\varrho, \varrho_{\ast})\|D_{W}\mathcal{F}(\varrho)\|_{L^{2}(\varrho)} - \frac{\lambda}{2}W_{2}(\varrho, \varrho_{\ast})^2}_{\text{HWI inequality}}
\end{equation*}
and transport inequality,
\begin{equation*}
    \frac{\lambda}{2}W_{2}(\varrho, \varrho)^{2} \le \mathcal{F}(\varrho) - \mathcal{F}(\varrho_{\ast})
\end{equation*}
with a positive constant $\lambda$. Here the quantity $\|D_{W}\mathcal{F}(\varrho)\|_{L^{2}(\varrho)}$ is the metric gradient or generalized Fisher information. The long-time asymptotics can be then deduced from the energy-dissipation
\begin{equation*}
    \frac{\mathrm{d} \big(\mathcal{F}(\varrho(t)) - \mathcal{F}(\varrho_{\ast})\big)}{\mathrm{d}t} = -\|D_{W}\mathcal{F}(\varrho(t))\|^{2}_{L^{2}(\varrho(t))} \le -2\lambda \big(\mathcal{F}(\varrho(t)) - \mathcal{F}(\varrho_{\ast})\big)
\end{equation*}
and Gr\"onwall inequality.

However, one of the main technical difficulties in extending the Bakry-\'Emery method to general settings is that global displacement convexity is typically established under convexity of confinement potentials and interaction kernels, which sometimes fall short when dealing with non-convex models. Instead some alternative methods would be used to study the long-time asymptotics.

For the classical Fokker-Planck equation with non-convex drift potential $V$,
\begin{equation*}
    \frac{\partial \varrho}{\partial t} - \Delta \varrho - \mathrm{div}(\varrho \nabla V) = 0,
\end{equation*}
there are rich theories supporting the long-time asymptotics \cite{bakry2014analysis, villani2008optimal}. One is because of its inherit connection with stochastic differential equations.
\begin{equation*}
    \mathrm{d}X(t) = -\nabla V(X(t)) \ud t + \sqrt{2}\mathrm{d}B(t)
\end{equation*}
where $B$ is the Brownian motion, and $\varrho(t)$ corresponds to the density distribution of $X(t)$. Another is based on the spectral analysis of the linear evolution operator (see, for example, \cite{villani2006hypocoercivity}). Here, by introducing the equilibrium probability measure $\nu(x) = \frac{1}{Z}\exp(-V(x))$ where $Z$ is the normalization constant, the Fokker-Planck equation can be formulated as a weighted diffusion equation (see, for example \cite{deng2026homogenization} and the references therein).\footnote{It can be also viewed as a diffusion equation on a Riemmanian manifold conformal to the Euclidean metric \cite[Chapter 14]{villani2008optimal}.}
\begin{equation*}
    \frac{\partial (\nu(x)u)}{\partial t} - \mathrm{div}(\nu(x) \nabla u) = 0,\quad \text{ where } u(x, t) := \varrho(x, t)/\nu(x).
\end{equation*}
Hence, the problem of long-time asymptotics is transformed into the existence of a spectral gap of the Laplacian on the weighted space, or proving the weighted Poincare inequality
\begin{equation*}
    \int_{\mathbb{R}^d} |\nabla u|^2 \ud \nu \ge \lambda \int_{\mathbb{R}^d} |u|^2 \ud \nu,\quad \text{ for all } \int_{\mathbb{R}^d} u \ud \nu = 0,
\end{equation*}
and there are many criteria for weighted Poincare inequalities (see for example, \cite{bakry2008simple} and the overview paper \cite{bonnefont2016note}) where only a linear growth of $V$ at infinity is needed.

For the McKean-Vlasov equation,
\begin{equation*}
    \frac{\partial \varrho}{\partial t} - \Delta \varrho - \mathrm{div}(\varrho \nabla (V + W\ast\varrho)) = 0,
\end{equation*}
there are probabilistic approaches based on the stochastic differential equation interpretation, and the convexity condition on $W$ can be relaxed to convexity at infinity \cite{cattiaux2008probabilistic, bolley2013uniform, durmus2024sticky}. The mechanism in these relaxation towards equilibrium without the convexity of $V, W$ is closely related to metastability in statistical mechanics, which describes the escape of particles from local minimum to global minimum via a thermally activated fluctuation.

\subsection{Non-convexity phenomenon}The problem becomes more challenging for nonlinear diffusion with non-convex kernels, for which no direct stochastic counterpart is generally available and several new phenomena may arise.

\textit{Finite time blow-up.} For aggregation-diffusion equation~\eqref{0-1} with singular kernels such as Newton kernel and Riesz kernel, finite-time blow-up may occur when strong singular attraction overwhelms diffusion and concentrates a substantial amount of mass in a finite time. An classical example is the supercritical Keller-Segel equation in dimension two, where solutions started from initial values with mass surmounting $8\pi$ will blow-up in finite time \cite{bedrossian2010intermediate, kim2012patlak, carrillo2019aggregation}.

\textit{Metastability.} For asymptotically flat kernels, the solution may present metastable phenomenon if there are multiple bumps separated far from each other \cite[Figure 5]{bailo2023aggregation}. In this case, each bump will relax very fast towards a distribution that resembles a steady state, followed by a time where bumps contract each other imperceptibly, and remain trapped in this slow motion mode for a long-time.

\subsection{Background}For our considered problem~\eqref{0-1}, there are rich results supporting the properties of steady states, which are critical points of the energy functional~\eqref{0-2-energy}. 
Different from linear diffusion \cite{carrillo2019existence}, degenerate diffusion exhibits finite propagation, and in this case, any strictly increasing kernel can produce a compactly supported ground-state distribution, without requiring growth at infinity.

In the ground-breaking work \cite{carrillo2016nonlinear}, Carrillo--Hittmeir--Volzone--Yao established the radial-symmetry of steady states of the aggregation-diffusion energy~\eqref{0-2-energy}, and conclude the uniqueness\footnote{Hereinafter, the uniqueness refers to being unique up to translation.} of steady-states based on \cite{carrillo2015ground, kim2012patlak}. The qualitative convergence of solutions to~\eqref{0-1} is then deduced from LaSalle's invariance principle based on the uniqueness of steady state, without relying on the displacement convexity. For more references and results concerning steady states we refer the reader to the previous work \cite{deng2026quantitative}.

For quantitative convergence of~\eqref{0-1}, Chayes--Kim--Yao \cite{Chayes} proved exponential convergence to the uniform equilibrium for degenerate aggregation–diffusion equations on torus with $1 < m \le 2$. In \cite{kim2012patlak}, Kim--Yao prove the exponential convergence with Newton kernel in the subcritical regime for radial decreasing solutions, and $\Delta W \ge 0$ necessary for the preservation of radial decreasing along the flow. In \cite{shu2021equilibration}, Shu provided an algebraic convergence towards equilibrium for some kernels with algebraic decay. 

\subsection{Main result and discussion}
Based on our previous spectral gap theorem \cite{deng2026quantitative} and the non-degeneracy results of Zhang \cite{zhang2022}, we establish a sharp quantitative convergence of~\eqref{0-1} in dimensional one for $m > 2$ with a class of sublinear radial increasing attraction kernels. Here the sharpness refers to the exponential convergence can not be improved.

Throughout the paper, we fix the mass center of the initial value at $0$, i.e., $\int_{\mathbb{R}} x\ud\varrho_0 = 0$, without loss of generality. By the symmetry of the aggregation kernel $W$, the mass center is conserved, and we have
\begin{equation*}
    \int_{\mathbb{R}}\, x\ud\varrho(t) = 0
\end{equation*}
as long as the solution exists. We consider the aggregation-diffusion equation~\eqref{0-1} on one-dimensional domain in the following. Our main result is given as follows.
\begin{theorem}\label{theorem_1_main}For kernel $W$ satisfying Assumption~\ref{assumption-1}, and initial value $\varrho_{0}$ satisfying Assumption~\ref{assumption-2}, there exists a unique bounded, compactly supported solution $\varrho(t)$ to~\eqref{0-1} started from $\varrho_{0}$.

Let $\varrho_{\ast}$ be the steady state centered at $0$. For any $\varepsilon > 0$, there exists $t_{0}(\varepsilon) > 0$ such that,
\begin{equation*}
    W_{2}(\varrho(t), \varrho_{\ast}) \lesssim \exp\big((\varepsilon - \Lambda_{R}) t\big).
\end{equation*}
for all $t \ge t_{0}(\varepsilon)$, where $\Lambda_{R} := \inf_{0 < |x| \le 2R} W^{\prime}(x)/x$ is the least attraction constant.
\end{theorem}

The main obstruction in establishing the quantitative convergence of general dimensions is that free boundary regularities for general dimensions (as that in the Fokker-Planck equation \cite{kim2021}) are currently unavailable for aggregation-diffusion equation~\eqref{0-1}. We expect that the approach developed here can also be applied to radially symmetric densities in higher dimensions, for attraction kernel $W$ with growth rate $\mathcal{N}_{d}(x)^{\alpha}$ for $\alpha > \frac{1}{2}$ and $\mathcal{N}_{d}$ is the Newton kernel of dimension $d$.

For $\log$ growth kernels or singular attractive Riesz kernels in dimension-one, the sharp qualitative convergence is still open to date. One central problem is to prove or disprove a uniform bounds on the evolving support (as that in Proposition~\ref{proposition-8-trapping}) if we started from a suitable compactly supported density. It requires a delicate balance between long-distance attraction and pressure gradient, at support boundaries.

\section*{Acknowledgment}
The main idea of the work came up during the author's study at the University of Minnesota. The motivating problem was given by Professor Jos\'e Carrillo during his curriculum talk on aggregation-diffusion equations at the University of Minnesota, in the November of 2025. The author remember that Professor Jos\'e Carrillo presented us a numerical example of aggregation-diffusion equation in dimension one, where $m = 2$ and $W$ is a singular attractive Riesz kernel, and he let us consider the scenario where two bumps are placed far away from each other, from which he questioned whether a quantitative convergence exists, or we merely have a qualitative convergence.

\textbf{AI usage}: 
\begin{enumerate}
    \item The uniform trapping in Proposition~\ref{proposition-8-trapping} is the joint work with ChatGPT 6.
    \item The $C^{1, \alpha}$ regularity of pressure at free boundary in Proposition~\ref{theorem:6-regularity} comes from a discussion with ChatGPT 6.
    \item All numerical results and illustrations are conducted purely by ChatGPT 6. 
\end{enumerate}
AI has been used to check the grammar and the consistency of formulas. The author is responsible for all contents. 

\section{Preparation tools}
In this section, we collect and construct some technical tools for proving quantitative convergence. For preliminaries on optimal transport, we refer the reader to the previous work \cite{deng2026quantitative}. One can regard the Wasserstein space $(\mathscr{P}_{2}(\mathbb{R}), W_{2})$ as an infinite dimensional Riemannian manifold where the tangent space is
\begin{equation*}
    T_{\varrho} \mathscr{P}_{2}(\mathbb{R}) = \overline{\big\{\nabla \phi : \phi \text{ is in } C_{c}^{\infty}(\mathbb{R})\big\}}_{L^{2}(\mathbb{R},\,\varrho)}
\end{equation*}
and the Riemannian metric is the Otto metric defined by
\begin{equation*}
    (X, Z)_{\varrho} = \int_{\mathbb{R}} X \cdot Z \ud \varrho.
\end{equation*}
This definition of Riemannian metric recovers $W_{2}$ as its Riemannian distance.

\subsection{Profiles of steady state}Steady states $\varrho_{\ast}$ are the critical points of the aggregation-diffusion energy functional
\begin{equation*}
    \mathcal{F}(\varrho) = \frac{1}{m - 1}\int_{\mathbb{R}}\, \varrho^{m}\ud x + \frac{1}{2}\int_{\mathbb{R}}\, \varrho \times (W\ast\varrho)\ud x,
\end{equation*}
defined such that $D_{W}\mathcal{F}(\varrho_{\ast}) = 0$. In our setting, we assume the model obeys the following.
\begin{assumption}\label{assumption-1}
We fix $m > 2$ and consider aggregation kernel $W$ that is smooth, symmetric, and strictly increasing in $[0, \infty)$, and it is non-degenerate at origin in the sense that
\begin{equation*}
    W''(0) > 0.
\end{equation*}
In addition, we assume that $W(x) \to \infty$ as $x \to \infty$, and either of the following
\begin{enumerate}
    \item $W(x)W^{\prime}(x) \to +\infty \text{ as } x \to +\infty$ and $W' \in C^{\infty}_b(\mathbb{R})$ when $\varrho_{0}$ is symmetric
    \item $W(x) W^{\prime}(x)^{\frac{m + 1}{m - 1}} \to +\infty \text{ as } x \to +\infty$ and $W' \in C^{\infty}_b(\mathbb{R})$
\end{enumerate}
These conditions give rise to the growth rate $\frac{1}{2} + \varepsilon$ if $\varrho_{0}$ is symmetric and compactly supported, or $\frac{1}{2} + \frac{1}{2m} + \varepsilon$ for a general compactly supported initial value, where $\varepsilon > 0$.
\end{assumption}

\begin{remark}
Once we have a posterior knowledge that the support of the solution $\varrho(t)$ is uniformly bounded in $B_{R}(0)$, the value of $W$ outside $|x| \ge 2R$ can be arbitrary (see, for example, \cite{yan2022sharp}).
\end{remark}

In this setting, Delgadino--Yan--Yao \cite{delgadino2022uniqueness} proved the uniqueness of steady state using a novel height interpolation technique, and by combining regularity results in \cite{carrillo2016nonlinear, carrillo2018ground}, they are able to prove the following characterization of equilibrium.
\begin{proposition}\label{proposition-3-steady}For aggregation-diffusion energy obeying~\eqref{0-2-energy}, there exists a unique steady state $\varrho_{\ast}$ centered at $0$, which is the minimizer of the energy. Furthermore $\varrho_{\ast}$ satisfy the following:
\begin{enumerate}
    \item $\varrho_{\ast}$ is continuous, compactly supported, symmetric, smooth inside support interior
    \item $\varrho_{\ast}$ is strictly radial decreasing, and non-degenerate at free boundary in the sense that $\varrho_{\ast}(x) \sim \delta(x, \mathbb{R} - \Omega_{\ast})^{\frac{1}{m - 1}}$ where $\Omega_{\ast}$ is the support of $\varrho_{\ast}$, and $\delta$ is the distance.
\end{enumerate}
In particular, due to the strict decreasing property, $\varrho_{\ast}$ is positive inside support interior.
\end{proposition}

Based on the height interpolation of Delgadino--Yan--Yao \cite{delgadino2022uniqueness} and the quantitative stability of interaction energy by Yan--Yao \cite{yan2022sharp}, we proved the following spectral gap theorem in \cite{deng2026quantitative}, by establishing transport inequalities which are weak characterization of eigenvalues.
\begin{theorem}\label{theorem-4-hessian}Let $D_{W}^{2}\mathcal{F}(\varrho_{\ast})$ be the Wasserstein Hessian or displacement Hessian, defined using the following second variation 
\begin{equation*}
    \int_{\mathbb{R}}\, (D_{W}^{2}\mathcal{F}(\varrho_{\ast})X, X) \ud \varrho_{\ast} = \frac{\mathrm{d}^2 \mathcal{F}((\mathrm{id} + \alpha X) \# \varrho_{\ast})}{\mathrm{d}\alpha^2},
\end{equation*}
we have
\begin{equation*}
    \int_{\mathbb{R}}\, (D_{W}^{2}\mathcal{F}(\varrho_{\ast})X, X) \ud \varrho_{\ast} \ge \Lambda_{R} \int_{\mathbb{R}}\, |X|^2 \ud \varrho_{\ast},\quad \text{ for all } \int_{\mathbb{R}}\, X \ud \varrho_{\ast} = 0.
\end{equation*}
where $\Lambda_{R}$ is the least attraction constant defined by $\Lambda_{R} = \inf_{0 < |x| \le 2R}W^{\prime}(x)/x$.
\end{theorem}

Our main strategy for proving the sharp quantitative convergence of the aggregation-diffusion equation~\eqref{0-1} is the principle of linearization based on this spectral gap theorem.

\subsection{Regularity}
In this part, we discuss parabolic estimates for $\varrho(t)$ in \cite{kim2021, zhang2022}, which are fundamental to the analysis of~\eqref{0-1}.

\begin{assumption}\label{assumption-2}Throughout this work, we assume that $\varrho_{0}$ is continuous and compactly supported. Let $\Omega_{0}$ be the support of $\varrho_{0}$ and $p_{0}(x) = \frac{m}{m - 1} \varrho_{0}(x)^{m - 1}$ be the pressure function.

We assume that $\Omega_{0}$ is a connected interval, $\varrho_{0}$ is strictly positive in the interior of $\Omega_{0}$, and $p_{0}$ satisfies the non-degeneracy condition at free boundary
\begin{equation*}
    \delta(x, \mathbb{R} - \Omega_{0})^{\gamma} \lesssim p_{0}(x),\quad \text{ for all } x \in \Omega,
\end{equation*}
with some exponent $0 < \gamma < 2$, where $\delta(x, \mathbb{R} - \Omega_{0})$ is the distance of $x$ to the free boundary.
\end{assumption}

By minimizing movements \cite{blanchet2008convergence, ambrosio2008gradient} or other methods in \cite{zhang2022}, one can construct weak solutions given as follows.
\begin{proposition}[existence]\label{proposition-4-existence}There exists a long-time weak solution $\varrho(t)$ to the aggregation-diffusion equation~\eqref{0-1} with initial value $\varrho_{0}$, which remains uniformly bounded in $L^{\infty}$.
\end{proposition}

We define
\begin{equation*}
    p(x, t) := \frac{m}{m - 1} \varrho(x, t)^{m - 1}
\end{equation*}
be the pressure function. By differentiating the above identity in time, and using the aggregation-diffusion law of evolution for $\dot{\varrho}$ in~\eqref{0-1}, the pressure evolves under the following quasilinear parabolic equation (see, for example, \cite[Equation (2.3)]{zhang2022})
\begin{equation}\label{2-6-pressure}
    \dot{p} = (m - 1)pp_{xx} + p_{x}^2 + p_{x}V(x, t) + (m - 1)p V_{x}(x, t),
\end{equation}
where $V := W^{\prime}\ast \varrho$ is a smooth drift vector field. Herein we use $\dot{p}$ to denote the time derivative and $p_{x}$ to denote the spatial derivative.

A parabolic maximum principle gives the following estimate (see, for example, \cite[Proposition 2.3]{zhang2022}).
\begin{proposition}[Aronson-Benilan]\label{proposition-5-AB} There exists a uniform constant $C$ such that $p_{xx}(x, t) \ge -\frac{1}{t} - C$ in the sense of distribution.
\end{proposition}

The following result is the consequence of \cite[Lemma 3.4]{zhang2022}, which states that interior vacuum can not be formed for attraction kernels.
\begin{proposition}[interior positivity]Started from initial value obeying Assumption~\ref{assumption-2}, the support of the solution $\varrho(t)$ is a closed interval:
\begin{equation*}
    \Omega(t) = [a(t), b(t)]
\end{equation*}
for some $C^{2, 1}_{\mathrm{loc}}((0, \infty))$ end-point functions $a(t)$ and $b(t)$.
\end{proposition}

The following Darcy's law \cite[Lemma 3.5]{zhang2022} gives the dynamics of free boundary (or support boundary) under the balances between pressure gradient and long-distance attraction of $W$.
\begin{proposition}[Darcy's law]\label{proposition-7-dirac}For every $t > 0$, the following limits exist,
\begin{equation*}
    p_{x}(a(t)^{+}, t) := \lim_{x \to a(t)^{+}} p_{x}(x, t),\quad D^{+}a(t) := \lim_{\tau \to 0^{+}} \frac{a(t + \tau) - a(t)}{\tau}
\end{equation*}
and the motion of free boundary obeys the Darcy's law:
\begin{equation*}
    D^{+}a(t) = -p_{x}(a(t)^{+}, t) - V(a(t),t),\quad \text{ where } V = W^{\prime} \ast \varrho
\end{equation*}
The same statement also applied to the right boundary $b(t)$.
\end{proposition}

\subsection{Long-time properties}\label{sec:2-3-long}
In this part, we prove important long-time properties for establishing the quantitative convergence. A numerical illustration of results in this section is given in Figure~\ref{fig:1-illustration}. Proofs of Propositions~\ref{proposition-8-trapping} and~\ref{proposition-4-long-time} are postponed to Section~\ref{sec:5-2-flow}. 

\begin{figure}
    \centering
    \includegraphics[width = 0.9\linewidth]{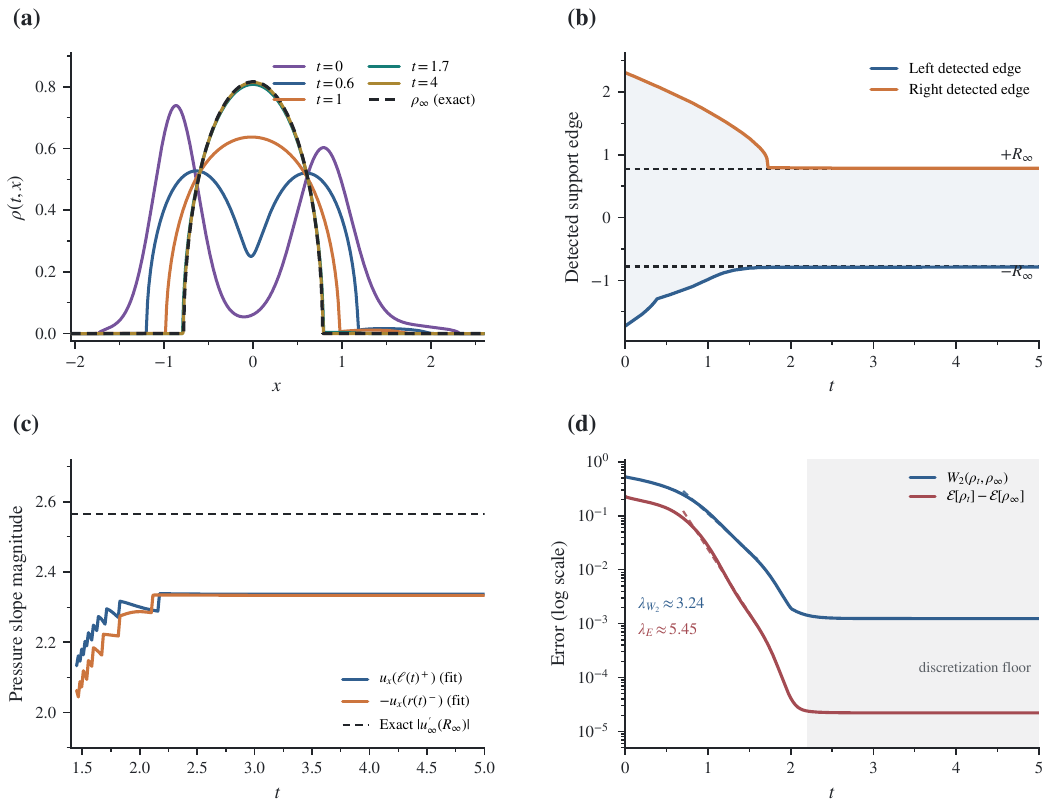}
    \caption{\small Numerical illustration of the uniform trapping, long-time non-degeneracy, and uniform interior positivity in Section~\ref{sec:2-3-long}. We consider a aggregation-diffusion with $m = 3$, and log-kernel $W(x) = \log(|x|)$ (which is outside of the assumption of the kernel, but worth considering in the future). Figure (a) shows that interior minimal will vanish in finite-time, and one can get a uniform interior positivity. Figure (b) shows the support of the solution remains uniformly bounded. Figure (c) records the gradient of pressure near free boundary, and showed that it stays uniformly positive. Figure (d) shows the quantitative convergence towards steady state and quantitative decay of energy.}\label{fig:1-illustration}
\end{figure}

\begin{proposition}[uniform trapping]\label{proposition-8-trapping}For aggregation kernels $W$ satisfying Assumption~\ref{assumption-1}, the bulk domain $\Omega(t) := [a(t), b(t)]$ is uniformly bounded
\end{proposition}

Based on this uniform trapping mechanism, one can then establish the following long-time non-degeneracy according to \cite[Theorem 4.1]{zhang2022}.
\begin{proposition}[uniform non-degeneracy]\label{proposition-4-long-time}Let $p(t)$ be the pressure function of $\varrho(t)$, we have
\begin{equation*}
    p_{x}(a(t)^{+}, t) \ge c(t_{0}),\quad -p_{x}(b(t)^{-}, t) \ge c(t_{0})
\end{equation*}
with a constant $c(t_{0})$ after a finite time $t_{0}$.
\end{proposition}

We next show the following $W_2$ qualitative convergence for aggregation-diffusion equation~\eqref{0-1}.
\begin{lemma}\label{prop-19-qual}Let $\varrho(t)$ be the solution of aggregation-diffusion equation~\eqref{0-1} started from $\varrho_{0}$, we have
\begin{equation*}
    \varrho(t) \to \varrho_{\ast} \text{ in } W_{2} \text{ as } t\to \infty.
\end{equation*}
where $\varrho_\ast$ is the unique steady state.
\end{lemma}
\begin{proof}
Since $\varrho(t)$ is the gradient flow of $\mathcal{F}(\varrho)$,
\begin{equation*}
    \frac{\mathrm{d}}{\mathrm{d}t} \mathcal{F}(\varrho(t)) = -\|D_{W}\mathcal{F}(\varrho(t))\|_{L^{2}(\mathbb{R}^d,\, \varrho(t))}^2 \le 0.
\end{equation*}
Hence $\mathcal{F}(\varrho)$ is the Lyponouv functional, i.e., $\mathcal{F}(\varrho(t)) \le \mathcal{F}(\varrho_{0})$ for all $t \ge 0$. By the uniform trapping in Proposition~\ref{proposition-8-trapping}, the time sequence $(\varrho(t))_{t \ge 0}$ is a bounded sequence in $\mathscr{P}_{\infty}(\mathbb{R}^d)$. By the compact embedding of $\mathscr{P}_{\infty}(\mathbb{R}^d)$ into $\mathscr{P}_{2}(\mathbb{R}^d)$ (see for example \cite[Chapter 6]{villani2008optimal}), $(\varrho(t))_{t \ge 0}$ is pre-compact in $\mathscr{P}_{2}(\mathbb{R}^d)$.

According to the uniqueness of steady state $\varrho_\ast$, for any subsequence $\varrho(t_{\ell})$, there further exists a (sub)subsequence $\varrho(t_{\ell(\nu)})$ converging to $\varrho_{\ast}$ as $\nu \to \infty$. Hence,
\begin{equation*}
    \varrho(t) \to \varrho_{\ast} \text{ in } W_{2},
\end{equation*}
or one can leverage LaSalle's invariance principle to get this qualitative convergence.
\end{proof}

We then improve the $W_{2}$ convergence to $L^{\infty}$ convergence using an interpolation between $\mathscr{P}_{2}(\mathbb{R})$ and $C^{\alpha}(\mathbb{R})$.
\begin{lemma}\label{lemma-3-interpolation}Given compactly supported continuous solution $\varrho(t)$ converging to $\varrho_{\ast}$, then we have
\begin{equation*}
    \|\varrho(t) - \varrho_{\ast}\|_{L^{\infty}} \lesssim W_{2}(\varrho(t), \varrho_{\ast})^{\frac{1}{2m - 1}}
\end{equation*}
where $\varrho_{\ast}$ is the steady state.
\end{lemma}

The proof of this lemma is shown in Appendix~\ref{ap:a}. Building on the $L^{\infty}$ convergence, we are able to get the following uniform interior positivity.
\begin{lemma}\label{proposition-12-up}For any compact subset $K$ of the interior of $\Omega_{\ast} \cap \Omega(t)$, there exists a time $t_{0}(K)$ such that, we have
\begin{equation*}
    \varrho(x, t) \ge c_{K},
\end{equation*}
for all $x \in K$ and $t \ge t_{0}(K)$, with a positive constant $c_{K}$
\end{lemma}

This uniform interior positivity is the direct result by combining the $L^{\infty}$ convergence in Lemma~\ref{lemma-3-interpolation} and the interior positivity of the steady state $\varrho_{\ast}$ in Proposition~\ref{proposition-3-steady}.

\section{Quantitative convergence in 1 dimension}
In this part, we prove the quantitative convergence based on the principle of linearization. 

\subsection{Regularity needed for linearization}
We establish the uniform upper bound of the $C^{1, \alpha}$ norm of $p(t)$, which is an important regularity estimate for conducting the linearization. Proofs of Propositions~\ref{theorem:6-regularity} and~\ref{theorem-4-optimal} are postponed to Section~\ref{sec:5-regularity}.
\begin{proposition}\label{theorem:6-regularity}Let $\Omega(t) = [a(t), b(t)]$ be the support of the density $\varrho(t)$, and $\Omega_{\ast} := [-R, R]$ be the support of the steady states $\varrho_{\ast}$. We define $A$ to be the affine interpolation given by
\begin{equation*}
    A: \Omega_{\ast} \to \Omega(t),\quad x \mapsto a(t) + \frac{x + R}{2R} (b(t) - a(t)).
\end{equation*}
Let $p_{A}(x, t) := p(A(x, t), t)$, after any short time $t_{0}$, we have
\begin{equation*}
    \|p_{A}(t)\|_{C^{1, \beta}(\Omega_{\ast})} \le C(t_{0})
\end{equation*}
for a constant $C(t_{0})$ and $0 < \beta \le 1$, and for all $t \ge t_{0}$.
\end{proposition}

Based on this pressure regularity, we show the following regularity of flow maps.
\begin{theorem}\label{theorem-4-optimal}Let $T_{\varrho_{\ast} \to \varrho(t)}$ be the optimal transport map from $\varrho_{\ast}$ to $\varrho(t)$, we have,
\begin{equation*}
    \partial_{x}T_{\varrho_{\ast} \to \varrho(t)} \ge c(t_{0}),\quad\|T_{\varrho_{\ast} \to \varrho(t)}\|_{C^{1, \beta}(\Omega_{\ast})} \le C(t_{0})
\end{equation*}
for positive constants $c(t_{0})$, $C(t_{0})$ and exponents $0 < \beta \le 1$, and all $t \ge t_{0}$ after a short time $t_{0}$.
\end{theorem}

The qualitative convergence $\varrho(t) \to \varrho_{\ast}$ in $W_{2}$ implies that $T_{\varrho_{\ast} \to \varrho(t)} \to \mathrm{id}$ in $L^{2}$. An interpolation between this $L_{2}$ convergence and the uniform $C^{1, \beta}$ bound of $T_{\varrho_{\ast} \to \varrho(t)}$ leads to the following.
\begin{proposition}
Let $T_{\varrho_{\ast} \to \varrho(t)}$ be the optimal transport map to $\varrho(t)$, we have
\begin{equation*}
    T_{\varrho_{\ast} \to \varrho(t)} \to \mathrm{id} \text{ in } C^{1}(\Omega_{\ast}),
\end{equation*}
where $T_{\varrho_{\ast} \to \varrho(t)}$ is the optimal transport map from $\varrho_{\ast}$ to $\varrho(t)$.
\end{proposition}

Our following proof of quantitative convergence relies on the flow map parameterization, which is popular for solving one-dimensional optimal transport problems \cite{blanchet2008convergence, matthes2014convergence, matthes2017convergent, jacobs2025}, due to the well-known and very handy increasing rearrangement of mass (see, for example, \cite[Chapter 1]{villani2008optimal}). In particular, for two measures $\mu$ and $\nu$, let $F_{\mu}$ and $F_{\nu}$ be their accumulation function, and $X_{\mu}$ and $X_{\nu}$ be the pseudo-inverse of $F_{\mu}$ and $F_{\nu}$, the optimal transport map and 2-Wasserstein distance are given by
\begin{equation*}
    T_{\mu \to \nu} = X_{\nu} \circ F_{\mu},\quad W_{2}(\mu, \nu) = \|X_{\mu} - X_{\nu}\|_{L^{2}[0, 1]}
\end{equation*}
One technical difficulty in the analysis of flow map parameterization is that diffusion needs to be controlled near the flow map boundaries, which has been solved by the regularity analysis in this part.

\subsection{Proof of quantitative convergence}\label{sec:3-2-regularity}
Based on this high-order convergence of pressure function and flow map, we proceed to show the quantitative convergence.
\begin{proof}[Proof of Theorem~\ref{theorem_1_main}]We parametrize solutions $\varrho(t)$ in the local neighborhood of $\varrho_{\ast}$ using transport map $X : \Omega_{\ast} \to \mathbb{R}$ as follows
\begin{equation*}
    \varrho(t) = (\mathrm{id} + X(t)) \# \varrho_{\ast}. 
\end{equation*}
where $\mathrm{id} + X(t) = T_{\varrho_{\ast} \to \varrho(t)}$ is the optimal transport map. Parameterizing the aggregation-diffusion in terms of this transport map, we are able to get
\begin{equation*}
    \begin{aligned}
        \frac{\partial X(t)}{\partial t} + D_{X}\mathcal{F}\big((\mathrm{id} + X(t))\# \varrho_{\ast}\big) = 0.
    \end{aligned}
\end{equation*}
$D_{X}$ is the first variation of the aggregation-diffusion energy with respect to the flow map. Building on the change-of-variable formula of push-forward, one can explicitly calculate
\begin{equation*}
    \begin{aligned}
        D_{X}\mathcal{F}((\mathrm{id} + X)\# \varrho_{\ast}) &= \frac{1}{\varrho_{\ast}(x)} \frac{\partial}{\partial x} \bigg[\varrho_{\ast}(x)^{m} \bigg[\frac{1}{(1 + X'(x, t))^{m}} - 1\bigg]\bigg]\\
        &\quad+ \int_{\Omega_{\ast}} W^{\prime}\big(x - \underline{x} + X(x) - X(\underline{x})\big) - W^{\prime}(x - \underline{x}) \ud \varrho_{\ast}(\underline{x}).
    \end{aligned}
\end{equation*}
We then linearize the gradient flow at $X = 0$ as follows,
\begin{equation}\label{4-1-grad}
    \frac{\partial X(t)}{\partial t} + D_{W}^2\mathcal{F}(\varrho_{\ast}) X(t) + \mathcal{N}(X(t)) = 0,
\end{equation}
where $D_{W}^{2}\mathcal{F}(\varrho_{\ast})$ is the Wasserstein Hessian at steady state $\varrho_{\ast}$ defined in Theorem~\ref{theorem-4-hessian}, and
\begin{equation*}
    \mathcal{N}(X) := D_{X}\mathcal{F}((\mathrm{id} + X)\# \varrho_{\ast}) - D_{W}^2\mathcal{F}(\varrho_{\ast}) X,
\end{equation*}
is the residual term. According to Theorem~\ref{theorem-4-hessian}, we have characterized spectrum of the Wasserstein Hessian as
\begin{equation*}
    (X, D_{W}^{2}\mathcal{F}(\varrho_{\ast}) X)_{\varrho_{\ast}} \ge \Lambda_{R} \|X\|_{L^{2}(\varrho_{\ast})}^2 \text{ for all } \int_{\Omega_{\ast}} X \ud \varrho_{\ast} = 0,
\end{equation*}
for a constant $\Lambda_{R}$. This spectral gap characterization is a transversal coercivity orthogonal to the mode of translation.

For the residual term, we will show in the next proof that
\begin{equation}\label{eq:4-2-residual}
    |(\mathcal{N}(X), X)_{\varrho_{\ast}}| \lesssim \|X\|_{W^{1, \infty}} \big((X, D_{W}^{2}\mathcal{F}(\varrho_{\ast}) X) + \|X\|_{L^{2}(\varrho)}^2\big).
\end{equation}
Now, leveraging the qualitative $C^{1}$ convergence of $X(t)$ to $0$, we have, for any $\varepsilon > 0$, there exists a time $t_{0}$ such that, for any $t \ge t_{0}$,
\begin{equation*}
    |(\mathcal{N}(X), X)_{\varrho_{\ast}}| \le \varepsilon \big((X, D_{W}^{2}\mathcal{F}(\varrho_{\ast}) X) + \|X\|_{L^{2}(\varrho)}^2\big).
\end{equation*}
Hence, multiplying both side of the linearized gradient flow~\eqref{4-1-grad} with $X(t)$ and integrating over the space $L^{2}(\Omega_{\ast}, \varrho_{\ast})$. Let $e(t) := \|X(t)\|_{L^{2}(\varrho_{\ast})}^2$ and sort the equitation, we have
\begin{equation*}
    \begin{aligned}
        \frac{1}{2}\frac{\mathrm{d}e(t)}{\mathrm{d}t} &\le \varepsilon \big((X, D_{W}^{2}\mathcal{F}(\varrho_{\ast}) X) + e(t)\big) - (X, D_{W}^{2}\mathcal{F}(\varrho_{\ast})X)_{\varrho_{\ast}}\\
        &\le \big(\Lambda_{R} - \varepsilon(1 + \Lambda_{R})\big) e(t).
    \end{aligned}
\end{equation*}
Finally, by using the Gr\"onwall inequality, 
\begin{equation*}
    \|X(t)\|_{L^{2}(\varrho_{\ast})} \le \|X(t_0)\|_{L^{2}(\varrho_{\ast})} e^{-\alpha (t - t_{0})} \lesssim e^{-\alpha t} \text{ for all } t \ge t_{0},
\end{equation*}
where $\alpha = \Lambda_{R} - \varepsilon(1 + \Lambda_{R})$. Hence, we have $W_{2}(\varrho(t), \varrho_{\ast}) = \|X(t)\|_{L^{2}(\varrho_{\ast})} \lesssim e^{-\alpha t}$.
\end{proof}

\begin{proof}[Proof of Equation~\eqref{eq:4-2-residual}]In the following, we check that the residual $\mathcal{N}$ can be made arbitrary small as long as $X'$ is small. We split the residual into two parts,
\begin{equation*}
    \mathcal{N} = \mathcal{N}_{D} + \mathcal{N}_{A}
\end{equation*}
namely residual from linearized diffusion $\mathcal{N}_{D}$, and residual from linearized aggregation $\mathcal{N}_{A}$.

For the diffusion part, one has
\begin{equation*}
    (\mathcal{N}_{D}(X), X)_{\varrho_{\ast}} = \int_{\Omega_{\ast}} -X' \times \big((1 + X')^{-m} - 1 + mX'\big) \ud \varrho^{m}_{\ast}.
\end{equation*}
By Taylor expansion
\begin{equation*}
    (1 + s)^{-m} - 1 = - m s + c_m s^2 + O(s^3) \text{ for a constant $c_m$}
\end{equation*}
and H\"older's inequality, we have
\begin{equation*}
    \big|(\mathcal{N}_{D}(X), X)_{\varrho_{\ast}}\big| = \bigg|\int_{\Omega_{\ast}} -X' \times \Big[c_m (X')^{2} + O\big((X')^{3}\big)\Big] \ud \varrho^{m}_{\ast}\bigg| \lesssim \|X'\|_{L^{\infty}} (D_{W}^{2}\mathcal{S}(\varrho_{\ast}) X, X)_{\varrho_{\ast}},
\end{equation*}
where $D_{W}^{2}\mathcal{S}$ is the Wasserstein Hessian of the internal energy $(D_{W}^{2}\mathcal{S}(\varrho_{\ast}) X, X)_{\varrho_{\ast}} := m\int_{\Omega_{\ast}} (X')^2 \ud \varrho_{\ast}^{m}$. For the aggregation part, one has
\begin{equation*}
    \begin{aligned}
        \mathcal{N}_{A}(X)(x) = \int_{\Omega_{\ast}} W'\big(x - \underline{x} + X(x) - X(\underline{x})\big) &- W'(x - \underline{x}) \ud \varrho_{\ast}(\underline{x})\\
        &- \int_{\Omega_{\ast}} W''(x - \underline{x}) \big(X(x) - X(\underline{x})\big) \ud \varrho_{\ast}(\underline{x})
    \end{aligned}
\end{equation*}
By using the Taylor expansion of the attraction kernel $W$, we are able to get (in the second equality, we have used the fact that $\int_{\Omega_{\ast}}X \ud \varrho_{\ast} = 0$, which holds true along the aggregation-diffusion flow):
\begin{equation*}
    |\mathcal{N}_{A}(X)(x)| \lesssim \int_{\Omega_{\ast}} |X(x) - X(\underline{x})|^2 \ud \varrho_{\ast} (\underline{x}) = |X(x)|^2 + \|X\|_{L^{2}(\varrho_{\ast})}^2.
\end{equation*}
Therefore, we have
\begin{equation*}
    |(\mathcal{N}_{A}(X), X)_{\varrho_{\ast}}| \le \int_{\Omega_{\ast}} |\mathcal{N}_{A}(X)| \times |X| \ud \varrho_{\ast} \lesssim \|X\|_{L^{\infty}} \times \|X\|_{L^{2}(\varrho_{\ast})}^2.
\end{equation*}
Summing the two residual parts $\mathcal{N}_{D}$ and $\mathcal{N}_{A}$ together, we can get the estimate shown in~\eqref{eq:4-2-residual}.
\end{proof}


\section{Proof of long-time properties}\label{sec:5-2-flow}
In this section, we establish the uniform trapping and non-degeneracy shown in Section~\ref{sec:2-3-long}. Here we denote $\dot{u}$ to be the time derivatives and $u_{x}$ the spatial derivative.

\begin{proof}[Proof of Proposition~\ref{proposition-8-trapping}]In this part, we prove the uniform trapping mechanism. The main idea is to estimate the inward attraction in the Darcy's law,
\begin{equation*}
    D^{+}a(t) = -p_{x}(a(t)^{+}, t) - V(a(t),t)
\end{equation*}
and show the inward attraction effects dominate the pressure effects and outward attraction over long distances. For simplicity, we work on controlling the left free boundary, and the right one can be controlled similarly by reflection. In particular, we define the mass accumulation function
\begin{equation*}
    F(x, t) := \int_{x}^{\infty} \varrho(\underline{x}, t) \ud \underline{x},
\end{equation*}
and the uniform trapping is equivalent to showing that, there exists a point $b_{\ast}$ independent of time, such that $F(x, t) = 0$ for all $x \ge b_{\ast}$. This goal can be achieved by building a compactly supported barrier $F_{\ast}(x)$ such that $F(x, t) \le F_{\ast}(x)$ for all time.

We first show, that a large amount of mass should be confined in neighborhoods of mass center when $m > 2$ and $W(x) \to \infty$ as $x \to \infty$. Since $\varrho(t)$ is a gradient flow, the energy $\mathcal{F}(\varrho)$ acts as the Lyapunov functional, and we have
\begin{equation*}
    \frac{1}{2}\int_{\mathbb{R}\times \mathbb{R}}\, W(x - \underline{x}) \varrho(x, t) \varrho(\underline{x}, t) \ud x\ud \underline{x} \le \mathcal{F}(\varrho(t)) \le \mathcal{F}(\varrho_0) =: E_{0}
\end{equation*}
We let $m(t)$ be the mass median\footnote{This is not the center of mass $\int_{\mathbb{R}} x \ud \varrho$ which is fixed at $0$.} of $\varrho(t)$ defined such that $\int_{-\infty}^{m(t)} \varrho(x, t) \ud x = \frac{1}{2}$, we have, according to the monotonicity of $W$:
\begin{equation*}
    \begin{aligned}
        \int_{m(t)}^{\infty} \int_{-\infty}^{m(t)}\, W(x &- \underline{x}) \varrho(x, t) \varrho(\underline{x}, t) \ud x\ud \underline{x}\\
        &\ge \frac{1}{2} \max\bigg\{\int_{m(t)}^{\infty} \, W(m(t) - \underline{x}) \varrho(\underline{x}, t) \ud \underline{x}, \int_{-\infty}^{m(t)}\, W(x - m(t)) \varrho(x, t) \ud x\bigg\}
    \end{aligned}
\end{equation*}
Hence, we have
\begin{equation}\label{2-6-center}
    \int_{\mathbb{R}} \, W\big(|x - m(t)|\big) \ud \varrho(t) \le 8 E_{0}.
\end{equation}
By further using the monotonicity of $W$, we are able to get
\begin{equation}\label{2-7-mass}
    \bigg[\int_{-\infty}^{m(t) - R} + \int_{m(t) + R}^{+\infty}\bigg] \varrho(x, t) \ud x \le \frac{8E_0}{W(R)}.
\end{equation}
As $W'(x) \to \infty$ when $x \to \infty$, there exists $R_{\ast}$ (independent of time) such that the fixed interval $[m(t) - R_\ast, m(t) + R_\ast]$ carries mass $\frac{1}{2}$, and we call the density located in this interval the \textit{stellar} of the mass.

\begin{remark}
The condition $W'(x) \to \infty$ when $x \to \infty$ rules out the dichotomy case of the concentration compactness theorem. As long as $m(t)$ is uniformly bounded, we can rule out the vanishing case, and establish the concentration compactness.
\end{remark}

We next discuss two cases for proving the uniform trapping.\\
(symmetric case) We first consider the case where $\varrho(t)$ is symmetric, i.e., $\varrho(x, t) = \varrho(-x, t)$. Due to the symmetry of the aggregation kernel, this symmetry is preserved along the flow as long as the initial value is symmetric. Here we work the with condition $W' \in C^{\infty}_{b}(\mathbb{R})$ and $W'(x)W(x) \to \infty$ as $x \to \infty$.

In particular, the symmetry ensures $m(t) = 0$. Let $R_\ast$ be the radius such that $[-R_\ast, R_\ast]$ covers $\frac{1}{2}$ mass shown above. We leverage the half amount of mass to attract mass far from the stellar. For simplicity, we work with the right part of the density, the left part can be obtained by symmetry.

For each $x > R_{\ast}$, we define the least attraction effects of stellar as follows
\begin{equation*}
    K_{R_{\ast}}(x) = \inf_{(x - R_\ast, x + R_\ast)} W'.
\end{equation*}
The inward attraction can be estimated by
\begin{equation*}
    \begin{aligned}
        V(x, t) &= \int_{\mathbb{R}}\, W'(x - \underline{x}) \ud \varrho(\underline{x}, t)\\
        &\ge \underbrace{\tfrac{1}{2} K_{R_{\ast}}(x)}_{\text{attraction of stellar}} - \underbrace{\|W'\|_{L^{\infty}} \int_{x}^{\infty} \varrho(\underline{x}, t) \ud \underline{x}}_{\text{tiny attraction of far-field}} \ge \tfrac{1}{2} K_{R_{\ast}}(x) - \frac{8\|W'\|_{L^{\infty}}E_{0}}{W(x)},
    \end{aligned}
\end{equation*}
where the tiny attraction of far-field is estimated by~\eqref{2-7-mass}. As long as $W'(x)W(x) \to \infty$ as $x \to \infty$, we have
\begin{equation*}
    K_{R_{\ast}}(x) W(x) \to \infty \text{ as } x \to \infty.
\end{equation*}
Therefore, there exists a frontier $\alpha > R_{\ast}$ such that
\begin{equation*}
    V(x, t) \ge V_{0}(x) := \tfrac{1}{4} K_{R_{\ast}}(x) \text{ when } x \ge \alpha.
\end{equation*}
Since $W'$ is uniformly bounded, we have $W(R) \lesssim R$. On the other hand, $W'(R)W(R) \to \infty$ implies that $R^{-1} \lesssim W'(R)$. Hence, we have $\int_{\alpha}^{\infty}V_{0}(x) \ud x = \infty$.

Let $\mathcal{B}(x) = \int_{\alpha}^{x} V_{0}(\underline{x}) \ud \underline{x}$, we build the following barrier
\begin{equation*}
    \varrho_{\mathcal{B}, M}(x) = \big[(1 - \tfrac{1}{m})(M - \mathcal{B}(x))_+\big]^{\frac{1}{m - 1}}
\end{equation*}
for $x \ge \alpha$. We choose $M$ large enough such that $\int_{\alpha}^{\infty} \varrho_{\mathcal{B}, M}(x) \ud x > 1$. Since $\mathcal{B}(\infty) = \infty$, $\varrho_{\mathcal{B}, M}(x)$ is always compactly supported for any $M$, and is solves the equilibrium equation
\begin{equation*}
    \frac{\partial \varrho_{\mathcal{B}, M}(x)^m}{\partial x} + V_{0} \varrho_{\mathcal{B}, M}(x) = 0.
\end{equation*}
Now, let $F(x, t) := \int_{x}^{\infty} \varrho(\underline{x}, t) \ud \underline{x}$ be the accumulation function. It obeys the law (for $x \ge \alpha$):
\begin{equation*}
    \frac{\partial F(x, t)}{\partial t} = m(-F_{x})^{m - 1}F_{xx} + V(x, t) F_{x} \le m(-F_{x})^{m - 1}F_{xx} + V_0(x) F_{x},
\end{equation*}
while barrier $F_{\mathcal{B}}(x) = \int_{x}^{\infty} \varrho_{\mathcal{B}, M}(\underline{x}, t) \ud \underline{x}$ solves the right hand side exactly. Hence $F_{\mathcal{B}}(x)$ is a super-solution of the parabolic equation, and therefore,
\begin{equation*}
    F(x, t) \le F_{\mathcal{B}}(x).
\end{equation*}
Let $x_{0}$ be the unique point such that $F_{\mathcal{B}}(x_0) = 0$ in $(\alpha, \infty)$, we then have $\varrho(x, t) = 0$ for $x \ge x_{0}$. By symmetry, $\varrho(x, t) = 0$ for all $x \le -x_{0}$, and we have the uniform trapping.

~\\
(general case) We prove the uniform trapping for asymmetric data. Here we let
\begin{equation*}
    K(R) = \inf_{(R/2, 3R/2)} W'.
\end{equation*}
and work with the condition
\begin{equation}\label{2-10-psi}
    \Psi(R) := W(R/2) K(R)^{\frac{m + 1}{m - 1}} \to \infty \text{ as } R \to \infty.
\end{equation}

We will use the following lemma which states that the center of stellar remains distant from the free boundary.
\begin{lemma}\label{lemma-10-media}If $\varrho(t)$ is supported within $[-R, R]$, for sufficiently large $R$, one has $|m(t)| \le \frac{1}{4}R$.
\end{lemma}

We argue by contradiction based on this lemma. Assume that the support of $\varrho(t)$ reaches a very far regions. Let $t_{R}$ be the first time that one free boundary hits $\partial B_R$.\footnote{$t_{R}$ is always finite due to the finite propagation of degenerate diffusion.} Without loss of generality, we assume it is the right free boundary. By Lemma~\ref{lemma-10-media}, $|m(t_R)| \le \tfrac{1}{4}R$ at $t_{R}$. If $R$ is large, by~\eqref{2-7-mass}, at least $\frac{1}{2}$ mass is contained in $[-\frac{1}{2}R, \frac{1}{2}R]$.

Suppose that the right free boundary is moving outward, by the Darcy's law in Proposition~\ref{proposition-7-dirac},
\begin{equation*}
    D^{+}b(t) = -p_{x}(b(t)^{-}, t) - V(b(t),t) > 0.
\end{equation*}
At $(R, t_R)$, the attraction force is
\begin{equation*}
    V(R, t_R) = \int_{-\infty}^{R} W'(R - x) \varrho(x, t_R) \ud x \ge \tfrac{1}{2} K(R)
\end{equation*}
because $\varrho(t)$ is vacuum outside $[-R, R]$. This means that
\begin{equation*}
    -p_{x}(R^{-}, t_R) \ge \tfrac{1}{2} K(R).
\end{equation*}
On the other hand, by using $p(R, t_R) = 0$, and Aronson-Benilan estimate $p_{xx} \ge C_{\mathrm{AB}}$ in Proposition~\ref{proposition-5-AB} for a uniform constant $C_{\mathrm{AB}}$, we can use a Taylor expansion of pressure to get
\begin{equation}\label{5-4-boundary}
    p(R - s, t_{R}) \ge -p_{x}(R^{-}, t_R)s - \tfrac{1}{2}C_{\mathrm{AB}}s^2 \ge -p_{x}(R^{-}, t_R)s/2
\end{equation}
for $s \le s_0 := -p_{x}(R^{-}, t_R)/C_{\mathrm{AB}}$.\footnote{The pressure gradient $p_x$ is uniformly bounded, which is one of the main results in \cite{zhang2022}.} Hence, to break through the attraction force, one gains a pressure layer. Converting the pressure inequality back into density,
\begin{equation*}
    \int_{R - s_0}^R \varrho(x, t_R) \ud x \ge C_m (-p_{x}(R^{-}, t_R))^{\frac{m + 1}{m - 1}} \ge \widehat{C}_{m} K(R)^{\frac{m + 1}{m - 1}}
\end{equation*}
with uniform constants $C_m$, $\widehat{C}_{m}$, and $s_0$. Hence, when $R$ is large, the domain of integration on the left hand side is completely $R/2$-distant to $m(t)$. From~\eqref{2-7-mass}, we have
\begin{equation*}
    \frac{8E_0}{W(R/2)} \ge \int_{R - s_0}^R \varrho(x, t_R) \ud x \ge \widehat{C}_{m} K(R)^{\frac{m + 1}{m - 1}},
\end{equation*}
which means that
\begin{equation*}
    \Psi(R) \le C
\end{equation*}
for a uniform constant $C$, which contradicts (or is ruled out by) to the condition~\eqref{2-10-psi}.
\end{proof}

\begin{remark}The main intuition for the computation of the third scenario is that, when the pressure gradient is sharp, according to Darcy's law, the bulk domain expands. If we ignore the attraction effect, the pressure will relaxes eventually to $0$. On the other hand, under attraction, if the pressure reaches some small pressure level below, it may balances the attraction at some time, and the bulk domain will stop evolving.

However, the main obstruction is that (when the attraction kernel is weak, for example, a negative Gaussian), when the pressure relaxes, free boundaries moves outward while simultaneously making themselves subjected to weaker attraction. In this case, the solution may continue to evolve without stopping, and we do not have a uniform trapping.

\end{remark}

\begin{proof}[Proof of Lemma~\ref{lemma-10-media}]As $|m(t)| \le R$, we have $|x - m(t)| \le 2R$ for all $x$ in the support of $\varrho(t)$. Hence, we have,
\begin{equation*}
    \begin{aligned}
        |m(t)| &\le \int_{\mathbb{R}}\, |x - m(t)| \ud \varrho(x, t)\\
        &\le \int_{m(t)-\alpha}^{m(t)+\alpha} |x - m(t)| \ud \varrho(x, t) + \bigg[\int_{-\infty}^{m(t) - \alpha} + \int_{m(t) + \alpha}^{+\infty}\bigg] |x - m(t)| \ud \varrho(x, t)\\
        &\le \alpha\int_{m(t)-\alpha}^{m(t)+\alpha}\ud \varrho(x, t) + 2R\bigg[\int_{-\infty}^{m(t) - \alpha} + \int_{m(t) + \alpha}^{+\infty}\bigg] \ud \varrho(x, t)\\
        &\le \alpha + 2R\bigg[\int_{-\infty}^{m(t) - \alpha} + \int_{m(t) + \alpha}^{+\infty}\bigg] \ud \varrho(x, t)
    \end{aligned}
\end{equation*}
for any $\alpha > 0$. Let $\alpha = R/8$, by~\eqref{2-7-mass}, we have
\begin{equation*}
    |m(t)| \le \alpha + \frac{16RE_0}{W(\alpha)} = \bigg[\frac{1}{8} + \frac{16 E_0}{W(R/8)}\bigg]R
\end{equation*}
which is less than $R/4$ when $R$ is large, because $W(R/8) \to \infty$ as $R \to \infty$.
\end{proof}

\begin{proof}[Proof of Proposition~\ref{proposition-4-long-time}]The main idea behind the proof of uniform non-degeneracy is based on the fact that the attraction force is non-vanishing, and we can get a uniform lower bound of attraction by uniform trapping (Proposition~\ref{proposition-8-trapping}). Therefore, pressure at free boundary should be non-degenerate to balances the attraction, in order to stop the free boundary motion.

In particular, in the Darcy's law of free boundary (Proposition~\ref{proposition-7-dirac}),
\begin{equation*}
    D^{+}a(t) = -p_{x}(a(t)^{+}, t) - V(a(t),t),\quad D^{+}b(t) = -p_{x}(b(t)^{-}, t) - V(b(t),t)
\end{equation*}
where $V = W^{\prime} \ast \varrho$. By Proposition~\ref{proposition-4-existence}, there exists $M > 0$ such that $\|\varrho(t)\|_{L^{\infty}} \le M$, and we have
\begin{equation*}
    \int_{a(t) + \varepsilon_0}^{b(t)} \varrho(x, t) \ud x \ge \tfrac{1}{2} \text{ where } \varepsilon_{0} := 1/(2M)
\end{equation*}
Hence, at the left free boundary, the attraction force is given by
\begin{equation*}
    \begin{aligned}
        V(a(t), t) &= \int_{a(t)}^{b(t)}W'(a(t) - \underline{x}) \ud \varrho(\underline{x}, t) \le \int_{a(t) + \varepsilon_0}^{b(t)}W'(a(t) - \underline{x}) \ud \varrho(\underline{x}, t)\\
        &\le -\tfrac{1}{2}\Lambda
    \end{aligned}
\end{equation*}
where $\Lambda := \inf_{(\varepsilon_{0}, 2R_\ast)} W'$ where $2R_\ast$ is the uniform upper bound for $b(t) - a(t)$, according to the uniform trapping in Proposition~\ref{proposition-8-trapping}. By reflection symmetry, we also have $V(b(t), t) \ge \tfrac{1}{2}\Lambda$.

For simplicity, we work with the left free boundary. According to \cite[Theorem 4.1]{zhang2022}, one has the free boundary non-degeneracy (after any short time $t_0$)
\begin{equation}\label{4-5-pressure}
    p_{x}(a(t)^{+}, t) \ge e^{-\sigma(t - t_{0})} p_{x}(a(t_0)^{+}, t_0),
\end{equation}
for all $t_{0} \le t \le t_{1}$ with a positive constant $\sigma = \sigma(t_{0})$. Hence, if a boundary slope becomes very small, it must also have been small throughout a preceding fixed amount of time.

And in this fixed amount of time, the attraction dominates the motion of free boundary. We argue that $p_x$ can not be very small by contradiction. Based on~\eqref{4-5-pressure}, consider the following backward in time control:
\begin{equation}\label{4-4-pressure}
    p_{x}(a(s)^{+}, s) \le e^{\sigma(t - s)} p_{x}(a(t)^{+}, t)
\end{equation}
for $t_0 \le s \le t$. Suppose that for some $t \ge t_0 + \tau$,
\begin{equation*}
    p_{x}(a(t)^{+}, t) \le \tfrac{1}{4}\Lambda e^{-\sigma \tau}
\end{equation*}
By using~\eqref{4-4-pressure}, we have $p_{x}(a(s)^+, s) \le \frac{1}{4}\Lambda$ for all $t - \tau \le s \le t$, we have
\begin{equation*}
    D^+a(t) = -p_{x}(a(s)^+, s) - V(a(t), t) \ge \tfrac{1}{4}\Lambda, \text{ in } [t - \tau, t].
\end{equation*}
Integrating free boundary vecolcity bound from $t-\tau \to t$, we are able to get
\begin{equation*}
    a(t) - a(t - \tau) \ge \tfrac{1}{4}\Lambda \tau
\end{equation*}
On the other hand, by uniform trapping, we have $a(t) - a(t - \tau) \le R_{\ast}$ for a constant $R_\ast$. Hence, taking $\tau > \tau_0 = 4R_\ast/\Lambda$ leads to a contradiction.

It implies that $p_x(a(t)^+, t)$ is uniformly bounded from below when $t \ge t_0 + \tau$, the same applies to the right free boundary.
\end{proof}

\section{Proofs of regularity results}\label{sec:5-regularity}
In this part, we proof the regularity needed for linearization in Section~\ref{sec:3-2-regularity}. 
\begin{proof}[Proof of Proposition~\ref{theorem:6-regularity}]For \textit{interior regularity}, we leverage the De Giorgi-Nash parabolic estimates in non-divergence form due to Krylov--Safonov \cite{krylov1981certain}, to prove a $C^{1, \alpha}$ interior regularity of $p$. Let $u = p_{x}$. Differentiating the pressure equation~\eqref{2-6-pressure} in $x$, we have
\begin{equation*}
    \dot{u} = (m - 1) p(x) u_{xx} + \big((m + 1)u + V(x)\big) u_x + mu V_{x}(x, t) + (m - 1)u V_{xx}(x, t).
\end{equation*}
By Lemma~\ref{proposition-12-up}, the pressure function $p$ is strictly positive in a compact subset $K$ of $\Omega_{\ast} \cap \Omega(t)$ after some time $t_{0}$. Furthermore, according to the $L^{\infty}$ convergence of $\varrho(t)$ to $\varrho_{\ast}$, the set $K$ can be made close to $\Omega(t)$ when $t_{0}$ is large. Hence, the parabolic PDE becomes uniformly parabolic after entering the following region $K \times [t_{0}, \infty)$. We let $U$ be an open subset of $K$, and $\tau > t_{0}$
\begin{equation*}
    Q := U \times [\tau, \infty).
\end{equation*}
The Krylov--Safonov estimate of parabolic PDE in non-divergence form \cite{krylov1981certain} then gives $\|u\|_{C^{0, \beta}(Q)}$ is uniformly bounded. Hence, we have
\begin{equation*}
    \|p\|_{C^{1, \beta}(Q)} \le C
\end{equation*}
for some constant $C$ depending on $U, \tau$, and a H\"older exponent $0 < \beta \le 1$.

For \textit{boundary regularity}, we work with the right free boundary, the left boundary can be obtained by reflection symmetry. By the second-order expansion of pressure layer in~\eqref{5-4-boundary}, we have,
\begin{equation*}
    c_{0} - \lambda s \le -p_{x}(b(t) - s, t) \le c_1,
\end{equation*}
near the right free boundary after some time $t_{0}$, where $c_{0}$ is a lower bound for $-p_{x}(b(t)^{+}, t)$ which is positive due to the non-degeneracy of $p$, and $\lambda$ is a global positive constant.

We pick $\varepsilon_{0}$ such that $\varepsilon_{0}\lambda < \frac{1}{2}c_{0}$, and we have
\begin{equation*}
    \tfrac{1}{2}c_{0} \le -p_{x}(x, t) \le c_{1},\quad \text{ for all } b(t) - \varepsilon_{0} < x < b(t).
\end{equation*}
The derivative bound makes the pressure function $p(x, t)$ strictly decreasing near the free boundary. Hence, the pressure function $p$ is invertible in this region, and there exists a uniform pressure level $z_{0}$ such that we can define the inverse function $Z(t)$ as follows
\begin{equation*}
    p(Z(z, t), t) = z,\quad 0 < z < z_{0}.
\end{equation*}
Let $A(z, t) = -1/p_{x}(Z(z, t), t)$, we have, due to the non-degeneracy of $p$,
\begin{equation}\label{5-2-bound}
    C_{0} \le A(z, t) \le C_{1}
\end{equation}
near the right free boundary $b(t)$, for positive constants $C_{0}$ and $C_{1}$. We next show that $A$ satisfies a parabolic equation. In particular, by differentiating the identity $p(Z(z, t), t)$, one has
\begin{equation*}
    -A = Z_{z},\quad -\dot{p}/p_{x} = \dot{Z},\quad p_{x} = -Z^{-1},\quad p_{xx} = -Z_{z}Z^{-3}.
\end{equation*}
Hence, the inverse function $Z$ obeys the following law
\begin{equation*}
    \dot{Z} = -(m - 1)zA^{-2}A_{z} + A^{-1} - V(Z, t) + (m - 1) zV_x(Z, t)A.
\end{equation*}
By differentiating the identity $X_{z} = -A$ in time, we are able to get
\begin{equation}\label{4-5-A}
    \begin{aligned}
        \dot{A} &= (m - 1)z^{-\alpha} \partial_{z}(z^{\alpha + 1}A^{-2}A_{z})\\
        &\quad- (m - 1)zV_x(Z, t)A_{z} - mV_x(Z, t)A + (m - 1)z V_{xx}(Z, t) A^2
    \end{aligned}
\end{equation}
where $\alpha = 1/(m - 1)$. Here the diffusion coefficient $A^{-2}$ is uniformly bounded up to free boundary due to~\eqref{5-2-bound}. Therefore, the leading order of the degeneracy near the free boundary is $z^{\alpha + 1}$.

We then use a radial lifting technique\footnote{This result is due to ChatGPT 6. It claims that this method comes from the classic Bessel operator \cite{wu2016}, which is related to the radial Laplacian.}. Consider the following change of variable:
\begin{equation*}
    s = 2z^{\frac{1}{2}},\quad \mathcal{A}(s, t) = A(s^2/4, t).
\end{equation*}
By taking this form of solution into~\ref{4-5-A}, we have
\begin{equation*}
    \dot{\mathcal{A}} = (m - 1) s^{-(2\alpha + 1)} \partial_{s}(s^{2\alpha + 1}\mathcal{A}^{-2} \mathcal{A}_{s}) - \frac{m - 1}{2}sV_x(Z, t)\mathcal{A}_{s} + N(s, t)
\end{equation*}
where $N$ is a uniformly bounded term. For $x \in \mathbb{R}^3$ and suppose that formally $|x| = s$, we can lift the function $\mathcal{A}$ into a radial function on $\mathbb{R}^3$ according to
\begin{equation*}
    \mathcal{B}(x, t) := \mathcal{A}(|x|, t).
\end{equation*}
Let $w(x) = |x|^{2\alpha - 1}$, the radial function $\mathcal{B}$ then obeys the following law of evolution:
\begin{equation}\label{5-4-b}
    w\dot{\mathcal{B}} = (m - 1) \mathrm{div}(w \mathcal{B}^{-2}\nabla \mathcal{B}) + w V_{\mathcal{B}} \cdot \nabla \mathcal{B} + w N_{\mathcal{B}}(x, t).
\end{equation}
It is a parabolic equation with the diffusion coefficient degenerating at origin. Here the coefficient $w$ has a exponent $2\alpha - 1 = (3 - m)/(m - 1)$ with $-1 < 2\alpha - 1 \le 1$ when $m \ge 2$, which belongs to the Muckenhoupt \(A_2(\mathbb R^3)\) class. Moreover, the origin has zero weighted capacity for this exponent, allowing the lifted equation to be extended weakly across $x = 0$. According to Chiarenza--Serapioni's Harnack inequality \cite{Chiarenza1984}, and we can get
\begin{equation*}
    |\mathcal{B}(x, t) - \mathcal{B}(x_{0}, t)| \lesssim |x - x_{0}|^{2\beta},
\end{equation*}
for some constant $0 < \beta \le 1$. Hence, when $x$ and $x_0$ is co-linear, let $s = |x|$ and $s_0 = |x_0|$, we have
\begin{equation*}
    |\mathcal{A}(s, t) - \mathcal{A}(s_{0}, t)| = |\mathcal{B}(x, t) - \mathcal{B}(x_{0}, t)| \lesssim |x - x_{0}|^{2\beta} = |s - s_{0}|^{2\beta},
\end{equation*}
uniformly up to the free boundary where $s$ or $s_0$ equals $0$, and therefore,
\begin{equation*}
    |A(z, t) - A(z_{0}, t)| = |\mathcal{A}(2z^{\frac{1}{2}}, t) - \mathcal{A}(2z^{\frac{1}{2}}_{0}, t)| \lesssim |z - z_{0}|^{\beta}.
\end{equation*}
Because $A = -1/p_x$, by the non-degeneracy of $A$, we can show that
\begin{equation*}
    |p_x(x, t) - p_x(x_{0}, t)| \lesssim |x - x_{0}|^{\beta}
\end{equation*}
near the free boundary. Combining the interior estimate above, we have shown the H\"older regularity over the whole domain.
\end{proof}

\begin{proof}[Proof of Theorem~\ref{theorem-4-optimal}]We denote the optimal transport map by $T(t) = T_{\varrho_{\ast} \to \varrho(t)}$, and let $\varrho_{A}(t)$ be the pull-back of the density $\varrho(t)$ using the affine map $A$ given by
\begin{equation*}
    A(x, t) = a(t) + \frac{x + R}{2R} (b(t) - a(t)).
\end{equation*}
By using change-of-variable formula, let $J(x, t) = \partial_{x}A(x, t)$, we have
\begin{equation*}
    \varrho_{A}(x, t) = J(x, t) \varrho(A(x, t), t).
\end{equation*}
We let $S(t)$ be the optimal transport map from $\varrho_{\ast}$ to $\varrho_{A}(t)$, because $A$ is monotone increasing, we have
\begin{equation*}
    T(t) = A(t) \circ S(t) \text{ and } \partial_{x}T(x, t) = J(x, t) \partial_{x}S(x, t).
\end{equation*}
Let $\alpha = 1/(m - 1)$. According to the definition of pressure function, we are able to get
\begin{equation*}
    \varrho_{\ast}(x) = C_{m}p_{\ast}(x)^{\alpha},\quad \varrho_{A}(x, t) = C_{m} J(x, t)p_{A}(x, t)^{\alpha}
\end{equation*}
with a constant $C_{m} = (1 - 1/m)^{\alpha}$.

In the following, we leveraging the regularity result (Proposition~\ref{theorem:6-regularity}) and non-degeneracy result (Proposition~\ref{proposition-4-long-time}) of pressure $p_{\ast}$ and $p_{A}$ to obtain the regularity of the flow map $S$, especially near the free boundary. For simplicity, we work with the left free boundary, and the estimate for right one can be obtained by reflection symmetry.

By the non-degeneracy of pressure at the left free boundary $x = - R$, we have
\begin{equation*}
    p_{A}^{\prime}(-R^{+}, t) \text{ exists and is non-zero.}
\end{equation*}
We net show that $p_{A}(x, t)/(x + R)$ is well-behaved in $[-R, 0]$. Consider the following identity,
\begin{equation*}
    p_{A}(x, t)/(x + R) = \int_{0}^{1} p_{A}^{\prime}(s(x + R) - R, t) \ud s.
\end{equation*}
By the H\"older continuity of pressure derivatives, we have
\begin{equation*}
    \begin{aligned}
        \big|p_{A}(x, t)/(x + R) - p_{A}(x_{0}, t)/(x_{0} + R)\big| &\le \int_{0}^{1} \big|p_{A}^{\prime}(s (x + R) - R, t) - p_{A}^{\prime}(s (x_{0} + R) - R, t)\big| \ud s\\
        &\le \|p_{A}^{\prime}\|_{C^{\alpha}} \int_{0}^{1} s^\alpha \ud s \times |x - x_{0}|^\alpha\\
        &= \frac{1}{1 + \alpha}\|p_{A}^{\prime}\|_{C^{\alpha}} \times |x - x_{0}|^\alpha.
    \end{aligned}
\end{equation*}
Hence, at least in a small region of the left free boundary, $p_{A}(x, t)/(x + R)$ is non-zero. Combining the interior positivity of the pressure function $p_{A}$, we are able to show, there exist positive uniform bounds $c_{0}$ and $c_{1}$, such that for all $x \in [-R, 0]$,
\begin{equation}\label{4-2-bound}
    c_{0} \le p_{A}(x, t)/(x + R) \le c_{1}.
\end{equation}
Consequently, there exist uniformly bounded, positive Hölder-continuous functions $u_{A}$ and $u_{\ast}$ such that
\begin{equation*}
    \varrho_{A}(x, t) = (x + R)^{\alpha} u_{A}(x, t),\quad \varrho_{\ast}(x) = (x + R)^{\alpha} u_{\ast}(x)
\end{equation*}
for all $x \in [-R, 0]$. Let $F_{A}$ and $F_{\ast}$ be the accumulation function of $\varrho_{A}$ and $\varrho_{\ast}$ respectively, we are able to deduce that
\begin{equation}\label{4-4-q}
    F_{A}(x, t) = (x + R)^{\alpha + 1} Q_{A}(x, t),\quad F_{\ast}(x) = (x + R)^{\alpha + 1} Q_{\ast}(x)
\end{equation}
for some bounded and Holder continuous functions $Q_{A}$ and $Q_{\ast}$ with $C_{0} \le Q_{A}, Q_{\ast} \le C_{1}$ for some constants $C_{0}$ and $C_{1}$.

Now, according to the definition of the optimal transport map $S(t)$ from $\varrho_{\ast}$ to $\varrho_{A}(t)$, we have
\begin{equation*}
    F_{A}(t) \circ S(t) = F_{\ast}.
\end{equation*}
By taking the form~\eqref{4-4-q} into the above identity, we have
\begin{equation}\label{4-6-s}
    \frac{S(x, t)}{x + R} = \bigg[\frac{Q_{\ast}(x)}{Q_{A}(S(x, t), t)}\bigg]^{1/(\alpha + 1)}.
\end{equation}
Since $C_{0} \le Q_{A}, Q_{\ast} \le C_{1}$, we have
\begin{equation}\label{4-5-s}
    K_{0}(x + R) \le S(x, t) \le K_{1}(x + R) \text{ for all } x \in [-R, 0],
\end{equation}
for some positive constants $K_{0}$ and $K_{1}$ around the left free boundary. Since $S(t)$ is the optimal transport map from $\varrho_{\ast}$ to $\varrho_{A}(t)$, we have, by the change-of-variable formula of transportation,
\begin{equation*}
    \partial_{x}S(x, t) = \frac{\varrho_{\ast}(x)}{\varrho_{A}(S(x, t), t)} = \frac{(x + R)^{\alpha}}{(S(x, t) + R)^{\alpha}} \frac{u_{\ast}(x)}{u_{A}(S(x, t), t)}.
\end{equation*}
By using~\eqref{4-5-s}, $\partial_{x}S(x, t)$ is uniformly bounded on $[-R, 0]$.

Since $Q_{\ast}, Q_{A}$ are $\beta$-H\"older continuous, then from~\eqref{4-6-s}, $S(x, t)/(x + R)$ is $\beta$-H\"older continuous. From~\eqref{4-5-s}, we have $\partial_{x}S(x, t)$ is $\beta$-H\"older continuous. Finally, since $T(t) := A(t) \circ S(t)$ where $A$ is a uniformly bounded affine map, we have
\begin{equation*}
    T_{\varrho_{\ast} \to \varrho(t)} := T(t) \text{ is } C^{1, \beta}(\Omega_{\ast}) \text{ and } 0 < c_{0} \le T(x, t) \le c_{1}
\end{equation*}
with positive constants $c_{0}$ and $c_{1}$. By reflection symmetry, we can extend the domain from $[-R, 0]$ to $\Omega_{\ast} := [-R, R]$, which concludes the proof.
\end{proof}

\appendix
\section{Technical lemmas}\label{ap:a}
\begin{proof}[Proof of interpolation inequality in Lemma~\ref{lemma-3-interpolation}]
Based on the inequality $W_{1}(\varrho(t), \varrho_{\ast}) \le W_{2}(\varrho(t), \varrho_{\ast})$ (see, for example, \cite[Chapter 6]{villani2008optimal}), we prove that 
\begin{equation*}
    \|\varrho(t) - \varrho_{\ast}\|_{L^{\infty}} \lesssim W_{1}(\varrho(t), \varrho_{\ast})^{\frac{1}{2m - 1}},
\end{equation*}
where $W_{1}$ has a dual characterization (Kantorovich-Rubinstein duality) that is related to $L^{\infty}$ norm as follows
\begin{equation}\label{2-3-KR}
    W_{1}(\varrho(t), \varrho_{\ast}) = \sup_{\mathrm{Lip}(\varphi) \le 1} \int_{\mathbb{R}} \varphi \ud \varrho(t) - \int_{\mathbb{R}} \varphi \ud \varrho_{\ast}.
\end{equation}
In the following, we find a suitable multiplier to deduce the lower bound of $W_{1}$ based on this sharp dual characterization. 

According to the Lipschitz continuity of the pressure function $p(t)$, we have
\begin{equation*}
    \begin{aligned}
        |\varrho(x, t) - \varrho(x_0, t)| &= C_{m} \big|p(x, t)^{\frac{1}{m - 1}} - p(x_0, t)^{\frac{1}{m - 1}}\big|\\
        &\le C_{m}|p(x, t) - p(x_0, t)|^{\frac{1}{m - 1}} \lesssim |x - x_0|^{\alpha}
    \end{aligned}
\end{equation*}
where $C_{m} = (1 - 1/m)^{1/(m - 1)}$ and $\alpha = 1/(m - 1)$. It shows that the solution $\varrho(t)$ pertains H\"older continuity, and the H\"older norm can be made independent of time. By the H\"older continuity of $\varrho_{\ast}$, we have
\begin{equation*}
    \|\varrho(t) - \varrho_{\ast}\|_{C^{\alpha}(\mathbb{R})} \le C_{m, \tau},
\end{equation*}
with a constant $C_{m, \tau}$. Let $a = \|\varrho(t) - \varrho_{\ast}\|_{L^{\infty}}$, without loss of generality, suppose that we can find a point $x_{0}$ such that
\begin{equation*}
    a = \varrho(x_{0}, t) - \varrho_{\ast}(x_{0}).
\end{equation*}
Otherwise, if we switch the role between $\varrho(x_{0}, t)$ and $\varrho_{\ast}(x_{0})$, the following argument also applies. Now, the H\"older continuity of $\varrho(t) - \varrho_{\ast}$ provides a quantitative deviation from $a$, and we have
\begin{equation}\label{2-4-holder}
    \varrho(x, t) - \varrho_{\ast}(x) \ge \tfrac{1}{2}a,\quad \text{ for all } |x - x_{0}| \le \big(\tfrac{1}{2}a/C_{m, \tau}\big)^{1/\alpha}.
\end{equation}
Consider the following multiplier in the Kantorovich–Rubinstein form~\eqref{2-3-KR}
\begin{equation*}
    \varphi_{\ast}(x) = \big(\big(\tfrac{1}{2}a/C_{m, \tau}\big)^{1/\alpha} - |x - x_{0}|\big)_{+}
\end{equation*}
which is a triangle function supported within the region where~\eqref{2-4-holder} holds valid. Therefore, we have
\begin{equation*}
    W_{1}(\varrho(t), \varrho_{\ast}) \ge \int_{\mathbb{R}} \varphi_{\ast} \ud (\varrho(t) - \varrho_{\ast}) \ge \tfrac{1}{2}a \int_{\mathbb{R}} \varphi_{\ast} \ud x = \tfrac{1}{2}a \big(\tfrac{1}{2}a/C_{m, \tau}\big)^{2/\alpha}.
\end{equation*}
Sorting the above inequality, we are able to get
\begin{equation*}
    \|\varrho(t) - \varrho_{\ast}\|_{L^{\infty}} =: a \lesssim W_{1}(\varrho(t), \varrho_{\ast})^{\alpha/(\alpha + 2)} = W_{1}(\varrho(t), \varrho_{\ast})^{\frac{1}{2m - 1}},
\end{equation*}
which concludes the proof.
\end{proof}

\bibliographystyle{alpha}

\begin{thebibliography}{CCCFR26}

\bibitem[AGS08]{ambrosio2008gradient}
Luigi Ambrosio, Nicola Gigli, and Giuseppe Savar{\'e}.
\newblock {\em Gradient flows: in metric spaces and in the space of probability measures}.
\newblock Springer Science \& Business Media, 2008.

\bibitem[AKSG19]{arbel2019maximum}
Michael Arbel, Anna Korba, Adil Salim, and Arthur Gretton.
\newblock Maximum mean discrepancy gradient flow.
\newblock {\em Advances in neural information processing systems}, 32, 2019.

\bibitem[BBCG08]{bakry2008simple}
Dominique Bakry, Franck Barthe, Patrick Cattiaux, and Arnaud Guillin.
\newblock A simple proof of the poincar{\'e} inequality for a large class of probability measures.
\newblock 2008.

\bibitem[BCC08]{blanchet2008convergence}
Adrien Blanchet, Vincent Calvez, and Jos{\'e}~A Carrillo.
\newblock Convergence of the mass-transport steepest descent scheme for the subcritical {Patlak--Keller--Segel} model.
\newblock {\em SIAM Journal on Numerical Analysis}, 46(2):691--721, 2008.

\bibitem[BCGC23]{bailo2023aggregation}
Rafael Bailo, Jos{\'e}~A Carrillo, and David G{\'o}mez-Castro.
\newblock Aggregation-diffusion equations for collective behaviour in the sciences.
\newblock In {\em International Congress on Industrial and Applied Mathematics}, pages 177--200. Springer, 2023.

\bibitem[B{\'E}06]{bakry2006diffusions}
Dominique Bakry and Michel {\'E}mery.
\newblock Diffusions hypercontractives.
\newblock In {\em S{\'e}minaire de Probabilit{\'e}s XIX 1983/84: Proceedings}, pages 177--206. Springer, 2006.

\bibitem[Bed10]{bedrossian2010intermediate}
Jacob Bedrossian.
\newblock Intermediate asymptotics for critical and supercritical aggregation equations and patlak-keller-segel models.
\newblock {\em arXiv preprint arXiv:1009.6187}, 2010.

\bibitem[BGG13]{bolley2013uniform}
Fran{\c{c}}ois Bolley, Ivan Gentil, and Arnaud Guillin.
\newblock Uniform convergence to equilibrium for granular media.
\newblock {\em Archive for Rational Mechanics and Analysis}, 208(2):429--445, 2013.

\bibitem[BGL14]{bakry2014analysis}
Dominique Bakry, Ivan Gentil, and Michel Ledoux.
\newblock {\em Analysis and geometry of Markov diffusion operators}, volume 103.
\newblock Springer, 2014.

\bibitem[BJM16]{bonnefont2016note}
Michel Bonnefont, Ald{\'e}ric Joulin, and Yutao Ma.
\newblock A note on spectral gap and weighted poincar{\'e} inequalities for some one-dimensional diffusions.
\newblock {\em ESAIM: Probability and Statistics}, 20:18--29, 2016.

\bibitem[CACP25]{castin2025unified}
Val{\'e}rie Castin, Pierre Ablin, Jos{\'e}~Antonio Carrillo, and Gabriel Peyr{\'e}.
\newblock A unified perspective on the dynamics of deep transformers.
\newblock {\em arXiv preprint arXiv:2501.18322}, 2025.

\bibitem[CCCFR26]{chizat2026quantitative}
L{\'e}na{\"\i}c Chizat, Maria Colombo, Roberto Colombo, and Xavier Fern{\'a}ndez-Real.
\newblock Quantitative convergence of {Wasserstein} gradient flows of kernel mean discrepancies.
\newblock {\em arXiv preprint arXiv:2603.01977}, 2026.

\bibitem[CCV15]{carrillo2015ground}
Jos{\'e}~Antonio Carrillo, Daniele Castorina, and Bruno Volzone.
\newblock Ground states for diffusion dominated free energies with logarithmic interaction.
\newblock {\em SIAM Journal on Mathematical Analysis}, 47(1):1--25, 2015.

\bibitem[CCY19]{carrillo2019aggregation}
Jos{\'e}~A Carrillo, Katy Craig, and Yao Yao.
\newblock Aggregation-diffusion equations: dynamics, asymptotics, and singular limits.
\newblock In {\em Active Particles, Volume 2: Advances in Theory, Models, and Applications}, pages 65--108. Springer, 2019.

\bibitem[CDP19]{carrillo2019existence}
JA~Carrillo, MG~Delgadino, and FS~Patacchini.
\newblock Existence of ground states for aggregation-diffusion equations.
\newblock {\em Analysis and applications}, 17(03):393--423, 2019.

\bibitem[CGM08]{cattiaux2008probabilistic}
Patrick Cattiaux, Arnaud Guillin, and Florent Malrieu.
\newblock Probabilistic approach for granular media equations in the non-uniformly convex case.
\newblock {\em Probability theory and related fields}, 140(1):19--40, 2008.

\bibitem[CHMV18]{carrillo2018ground}
Jos{\'e}~A Carrillo, Franca Hoffmann, Edoardo Mainini, and Bruno Volzone.
\newblock Ground states in the diffusion-dominated regime: Ja carrillo et al.
\newblock {\em Calculus of variations and partial differential equations}, 57(5):127, 2018.

\bibitem[CHVY16]{carrillo2016nonlinear}
Jos{\'e}~A Carrillo, Sabine Hittmeir, Bruno Volzone, and Yao Yao.
\newblock Nonlinear aggregation-diffusion equations: radial symmetry and long time asymptotics.
\newblock {\em arXiv preprint arXiv:1603.07767}, 2016.

\bibitem[CKY13]{Chayes}
Lincoln Chayes, Inwon Kim, and Yao Yao.
\newblock An aggregation equation with degenerate diffusion: Qualitative property of solutions.
\newblock {\em SIAM Journal on Mathematical Analysis}, 45(5):2995--3018, 2013.

\bibitem[CMV03]{carrillo2003kinetic}
Jos{\'e}~A Carrillo, Robert~J McCann, and C{\'e}dric Villani.
\newblock Kinetic equilibration rates for granular media and related equations: entropy dissipation and mass transportation estimates.
\newblock {\em Revista Matematica Iberoamericana}, 19(3):971--1018, 2003.

\bibitem[CS84]{Chiarenza1984}
Filippo~M Chiarenza and Raul~P Serapioni.
\newblock A harnack inequality for degenerate parabolic equations.
\newblock {\em Communications in partial differential equations}, 9(8):719--749, 1984.

\bibitem[DEGS24]{durmus2024sticky}
Alain Durmus, Andreas Eberle, Arnaud Guillin, and Katharina Schuh.
\newblock Sticky nonlinear sdes and convergence of mckean--vlasov equations without confinement.
\newblock {\em Stochastics and Partial Differential Equations: Analysis and Computations}, 12(3):1855--1906, 2024.

\bibitem[Den26]{deng2026quantitative}
Yunhong Deng.
\newblock Quantitative stability of aggregation-diffusion under local attraction.
\newblock {\em arXiv preprint arXiv:2610.07264}, 2026.

\bibitem[DGW26]{deng2026homogenization}
Yunhong Deng, Yuan Gao, and Li~Wang.
\newblock Homogenization of time-discrete gradient flows.
\newblock {\em arXiv preprint arXiv:2607.25206}, 2026.

\bibitem[DM05]{denzler2005fast}
Jochen Denzler and Robert~J McCann.
\newblock Fast diffusion to self-similarity: complete spectrum, long-time asymptotics, and numerology.
\newblock {\em Archive for rational mechanics and analysis}, 175(3):301--342, 2005.

\bibitem[DMNV25]{di2025inexact}
Simone Di~Marino, Emanuele Naldi, and Silvia Villa.
\newblock Inexact jko and proximal-gradient algorithms in the {Wasserstein} space.
\newblock {\em arXiv preprint arXiv:2505.23517}, 2025.

\bibitem[DYY22]{delgadino2022uniqueness}
Matias~G Delgadino, Xukai Yan, and Yao Yao.
\newblock Uniqueness and nonuniqueness of steady states of aggregation-diffusion equations.
\newblock {\em Communications on Pure and Applied Mathematics}, 75(1):3--59, 2022.

\bibitem[FRF22]{fernandez2022continuous}
Xavier Fern{\'a}ndez-Real and Alessio Figalli.
\newblock The continuous formulation of shallow neural networks as {Wasserstein}-type gradient flows.
\newblock In {\em Analysis at Large: Dedicated to the Life and Work of Jean Bourgain}, pages 29--57. Springer, 2022.

\bibitem[GLPR25]{geshkovski2025mathematical}
Borjan Geshkovski, Cyril Letrouit, Yury Polyanskiy, and Philippe Rigollet.
\newblock A mathematical perspective on transformers.
\newblock {\em Bulletin of the American Mathematical Society}, 62(3):427--479, 2025.

\bibitem[Jac25]{jacobs2025}
Matt Jacobs.
\newblock Lagrangian solutions to the porous media equation and reaction diffusion systems.
\newblock {\em SIAM Journal on Mathematical Analysis}, 57(4):3806--3839, 2025.

\bibitem[JKO98]{JKO}
Richard Jordan, David Kinderlehrer, and Felix Otto.
\newblock The variational formulation of the fokker--planck equation.
\newblock {\em SIAM Journal on Mathematical Analysis}, 29(1):1--17, 1998.

\bibitem[KS81]{krylov1981certain}
Nicolai~V Krylov and Mikhail~V Safonov.
\newblock A certain property of solutions of parabolic equations with measurable coefficients.
\newblock {\em Mathematics of the USSR-Izvestiya}, 16(1):151--164, 1981.

\bibitem[KY12]{kim2012patlak}
Inwon Kim and Yao Yao.
\newblock The patlak--keller--segel model and its variations: properties of solutions via maximum principle.
\newblock {\em SIAM Journal on Mathematical Analysis}, 44(2):568--602, 2012.

\bibitem[KZ21]{kim2021}
Inwon Kim and Yuming~Paul Zhang.
\newblock Porous medium equation with a drift: free boundary regularity.
\newblock {\em Archive for Rational Mechanics and Analysis}, 242(2):1177--1228, 2021.

\bibitem[LOWY24]{li2024stochastic}
Qin Li, Maria Oprea, Li~Wang, and Yunan Yang.
\newblock Stochastic inverse problem: stability, regularization and {Wasserstein} gradient flow.
\newblock {\em arXiv preprint arXiv:2410.00229}, 2024.

\bibitem[LS26]{lascu2026saddle}
Razvan-Andrei Lascu and Taiji Suzuki.
\newblock From saddle points toward global minima: A newton-type method on {Wasserstein} space.
\newblock {\em arXiv preprint arXiv:2605.17963}, 2026.

\bibitem[LZZ19]{liu2019understanding}
Chang Liu, Jingwei Zhuo, and Jun Zhu.
\newblock Understanding mcmc dynamics as flows on the wasserstein space.
\newblock In {\em International conference on machine learning}, pages 4093--4103. PMLR, 2019.

\bibitem[McC97]{mccann1997convexity}
Robert~J McCann.
\newblock A convexity principle for interacting gases.
\newblock {\em Advances in mathematics}, 128(1):153--179, 1997.

\bibitem[MO14]{matthes2014convergence}
Daniel Matthes and Horst Osberger.
\newblock Convergence of a variational lagrangian scheme for a nonlineardrift diffusion equation.
\newblock {\em ESAIM: Mathematical modelling and numerical analysis}, 48(3):697--726, 2014.

\bibitem[MS17]{matthes2017convergent}
Daniel Matthes and Benjamin S{\"o}llner.
\newblock Convergent lagrangian discretization for drift-diffusion with nonlocal aggregation.
\newblock In {\em Innovative algorithms and analysis}, pages 313--351. Springer, 2017.

\bibitem[OR07]{otto2007slow}
Felix Otto and Maria~G Reznikoff.
\newblock Slow motion of gradient flows.
\newblock {\em Journal of Differential Equations}, 237(2):372--420, 2007.

\bibitem[Ott01]{otto2001geometry}
Felix Otto.
\newblock The geometry of dissipative evolution equations: the porous medium equation.
\newblock 2001.

\bibitem[OV00]{otto2000generalization}
Felix Otto and C{\'e}dric Villani.
\newblock Generalization of an inequality by talagrand and links with the logarithmic sobolev inequality.
\newblock {\em Journal of Functional Analysis}, 173(2):361--400, 2000.

\bibitem[Rig25]{rigollet2025mean}
Philippe Rigollet.
\newblock The mean-field dynamics of transformers.
\newblock {\em arXiv preprint arXiv:2512.01868}, 2025.

\bibitem[RVE22]{rotskoff2022trainability}
Grant Rotskoff and Eric Vanden-Eijnden.
\newblock Trainability and accuracy of artificial neural networks: An interacting particle system approach.
\newblock {\em Communications on Pure and Applied Mathematics}, 75(9):1889--1935, 2022.

\bibitem[San15]{santambrogio2015optimal}
Filippo Santambrogio.
\newblock {\em Optimal transport for applied mathematicians}, volume~87.
\newblock Springer, 2015.

\bibitem[SDD{\etalchar{+}}23]{sun2023discrete}
Haoran Sun, Hanjun Dai, Bo~Dai, Haomin Zhou, and Dale Schuurmans.
\newblock Discrete langevin samplers via wasserstein gradient flow.
\newblock In {\em International Conference on Artificial Intelligence and Statistics}, pages 6290--6313. PMLR, 2023.

\bibitem[Sei14]{seis2014long}
Christian Seis.
\newblock Long-time asymptotics for the porous medium equation: the spectrum of the linearized operator.
\newblock {\em Journal of Differential Equations}, 256(3):1191--1223, 2014.

\bibitem[Shu21]{shu2021equilibration}
Ruiwen Shu.
\newblock Equilibration of aggregation-diffusion equations with weak interaction forces.
\newblock {\em SIAM Journal on Mathematical Analysis}, 53(6):6562--6616, 2021.

\bibitem[SSMC26]{seo2026local}
Geuntaek Seo, Minseop Shin, Pierre Monmarch{\'e}, and Beomjun Choi.
\newblock Local exponential stability of mean-field langevin descent-ascent in {Wasserstein} space.
\newblock {\em arXiv preprint arXiv:2602.01564}, 2026.

\bibitem[Vil06]{villani2006hypocoercivity}
C{\'e}dric Villani.
\newblock Hypocoercivity.
\newblock {\em arXiv preprint math/0609050}, 2006.

\bibitem[Vil08]{villani2008optimal}
C{\'e}dric Villani.
\newblock {\em Optimal transport: old and new}, volume 338.
\newblock Springer, 2008.

\bibitem[WL20]{wang2020information}
Yifei Wang and Wuchen Li.
\newblock Information newton's flow: second-order optimization method in probability space.
\newblock {\em arXiv preprint arXiv:2001.04341}, 2020.

\bibitem[WYZ16]{wu2016}
Huoxiong Wu, Dongyong Yang, and Jing Zhang.
\newblock Oscillation and variation for semigroups associated with bessel operators.
\newblock {\em Journal of Mathematical Analysis and Applications}, 443(2):848--867, 2016.

\bibitem[YY22]{yan2022sharp}
Xukai Yan and Yao Yao.
\newblock Sharp stability for the interaction energy: Xukai yan \& yao yao.
\newblock {\em Archive for Rational Mechanics and Analysis}, 246(2):603--629, 2022.

\bibitem[Zha22]{zhang2022}
Yuming~Paul Zhang.
\newblock Free boundary regularity of the porous medium equation with nonlocal drifts in dimension one.
\newblock {\em Calculus of variations and partial differential equations}, 61(2):41, 2022.

\end{thebibliography}
\newcommand{\etalchar}[1]{$^{#1}$}

\end{document}